\documentclass{amsproc}

\usepackage{amssymb}

\newtheorem{thm}{Theorem}[section]
\newtheorem{lem}[thm]{Lemma}
\newtheorem{prop}[thm]{Proposition}
\newtheorem{cor}[thm]{Corollary}
\newtheorem{exs}[thm]{Examples}
\newtheorem{ex}[thm]{Example}
\theoremstyle{definition}
\newtheorem{defn}[thm]{Definition}

\theoremstyle{remark}

\numberwithin{equation}{section}

\newcommand{\sbq}{\subseteq}
\newcommand{\spq}{\supseteq}

\newcommand{\vide}{\emptyset} 
\newcommand{\tbf}{\textbf}
\newcommand{\mbf}{\mathbf}

\newcommand{\inv}{^{-1}}

\newcommand{\C}{\text{C}}
  
\newcommand{\Rtwo}{\R\times\R}  
\newcommand{\betatwo}{\beta\R \times \beta\R} 
\newcommand{\betaNtwo}{\beta\N \times \beta\N}  
\DeclareMathOperator{\ann}{ann}
\DeclareMathOperator{\coz}{coz}
\DeclareMathOperator{\z}{z}
\newcommand{\R}{\mathbf R}
\newcommand{\N}{\mathbf N}
\newcommand{\Q}{\mathbf Q}
\newcommand{\Z}{\mathbf Z}

\newcommand{\cl}{\text{cl\,}}
\newcommand{\varep}{\varepsilon}
\newcommand{\res}{\raisebox{-.5ex}{$|$}}

\newcommand{\cx}{\text{C}(X)}

\newcommand{\rr}{\tbf{rr}}
\newcommand{\bd}{\text{\tbf{Fr\,}}}
\newcommand{\inte}{\ensuremath{\text{\tbf{int}}}}
\newcommand{\rro}{\le_{\text{rr}}}
\newcommand{\wrr}{\wedge_{\text{rr}}}
\newcommand{\T}{\text{\tbf{T}}}

\begin{document}
 \title[Tychonoff spaces and a ring theoretic order on $\cx$]{
Tychonoff spaces and a ring theoretic order on $\cx$}

 \author{W.D. Burgess}
\address{Department of Mathematics and Statistics\\ University of Ottawa, Ottawa, Canada, K1N 6N5}

\email{wburgess@uottawa.ca}
\thanks{The authors thank the referee for helpful observations as well as for suggestions about the presentation.  \\\indent DOI 10.1016/j.topol.2020.107250}

\author{R. Raphael}
\address{Department of Mathematics and Statistics\\ Concordia University, Montr\'eal, Canada, H4B 1R6}

\email{r.raphael@concordia.ca}

\begin{abstract} The reduced ring order (\rr-order) is a natural partial order on a reduced ring $R$ given by $r\rro s$ if $r^2=rs$. It can be studied algebraically or topologically in rings of the form $\cx$.  The focus here is on those reduced rings in which each pair of elements has an infimum in the \rr-order, and what this implies for $X$. A space $X$ is called \rr-good if $\cx$ has this property. Surprisingly both locally connected and basically disconnected spaces share this property. The \rr-good property is studied under various topological conditions including its behaviour under Cartesian products.  The product of two \rr-good spaces can fail to be \rr-good (e.g., $\beta \R\times \beta \R$), however, the product of a $P$-space and an \rr-good weakly Lindel\"of space is always \rr-good.  $P$-spaces, $F$-spaces and $U$-spaces play a role, as do Glicksberg's theorem and work by Comfort, Hindman and Negrepontis.   
 \end{abstract}
\maketitle

\noindent\tbf{Introduction.} 

In a reduced ring, a ring with no non-zero nilpotent elements, such as $\cx$, there is a partial order that generalizes the natural partial order on a boolean ring. The order relation is defined as $r\rro s$ if $r^2=rs$. The study of this order, here called the \emph{\rr-order} for the \emph{reduced ring order}, goes back at least to 1958 in \cite{S}.  Since then it has been studied at various times (see \cite{Ch} and \cite{B}, for example),  but most recently in \cite{BR1} and \cite{BR2}.  In these papers some of the most interesting examples and results are about rings of the form $\cx$. 

It is rare for a pair of elements in a reduced ring $R$ to have a supremum in the \rr-order and the most natural generalization of the boolean ring case is where the ring has, for every pair of elements $r,s\in R$, an infimum in the \rr-order, noted $r\wrr s$, i.e., when $R$ is a lower semi-lattice in the order. Such rings are called \emph{\rr-good}. A space $X$ is called \emph{\rr-good} if the ring $\cx$ is \rr-good. The theme of this paper is the study of spaces that are \rr-good and those that are not.

In the sequel, all topological spaces will be \emph{assumed to be Tychonoff spaces.}

Not all spaces are \rr-good but those that are form a surprisingly diverse family that includes locally connected spaces and those that are basically disconnected.  To find a topological characterization of \rr-good spaces would seem an unrealistic task but much can be said about them. There are connected spaces that are not \rr-good (see Theorem~\ref{hager}, below) and even connected, compact metric spaces that are not \rr-good (\cite[Theorem~3.5(2)]{BR1}).  

The paper is divided into five sections. The first gives some basic results and examples. The second deals with when a product of two spaces is \rr-good. If $X\times Y$ is \rr-good it is easy to see that $X$ and $Y$ are also \rr-good.  The converse, false in general, turns out to be a rich subject. 

If the real line $\R$ is partitioned into two complementary dense subspaces, neither can be \rr-good. The third section shows that $\R^2$ is quite different. Complementary dense subspaces of  the plane are found, one of which is \rr-good.

 Section~4 examines basically disconnected and $U$-spaces with an emphasis on separation properties and their consequences for  $\cx$.
 
 Section~5 looks at a sufficient condition, called the B-property (for boundary property), for a space to be \rr-good. Basically disconnected spaces that are not discrete do not have the B-property. Here it is shown how to find connected \rr-good spaces without the B-property.
 \\[-.5ex]

\setcounter{section}{1}
\setcounter{thm}{0}
\noindent\tbf{1. The definition of \rr-good spaces: basic properties and examples.}

To recall: a ring $\cx$ is partially ordered by the relation $f\rro g$ if $f^2=fg$. When $h\rro f$ and $h\rro g$ this is abbreviated to $h\rro f,g$. The following facts are obvious but are basic tools underlying many of the results used below. 

\begin{lem} \label{fundlem}  In a ring $\cx$, (1) if $f\rro g$ then $f$ and $g$ coincide on $\cl(\coz f)$; 

(2) if $h\rro f,g$ then $\coz h\sbq \z(f-g) \cap \coz f$.  \end{lem}

It is clear that a free union of \rr-good spaces is \rr-good since a product of \rr-good rings is \rr-good. The following proposition quotes results that show how \rr-good spaces can be found from a given \rr-good space.

\begin{prop} \label{recall}  (1) (\cite[Proposition 3.10]{BR1}) A cozero set in an \rr-good space is \rr-good. 

(2) (\cite[Proposition 3.9]{BR1}) The ring $\cx$ is \rr-good if and only if $\C^*(X)$ is \rr-good, i.e.,  a space $X$ is \rr-good if and only if $\beta X$ is \rr-good. \end{prop} 

The two main classes of examples of \rr-good spaces are summarized here.

\begin{exs}\label{basicexs} (1)~\cite[Theorem~3.5(1)]{BR1} If $X$ is a locally connected space then $X$ is \rr-good.

(2) \cite[before Ex.\,3.2]{NR} If $X$ is basically disconnected then $X$ is \rr-good. \end{exs}

The second case will be expanded upon in Section~4. 

It is not true that a quotient space of an \rr-good space has to be \rr-good. 

\begin{ex} \label{N*} The space $\beta\N$ is \rr-good but its quotient space  $\N\cup \{\infty\}$ (the one-point compactification) is not.  \end{ex}

\begin{proof} The space $\N\cup \{\infty\}$ is not \rr-good by \cite[Proposition~3.6]{BR1}, or see Lemma~\ref{nogood}, below.  \end{proof}

Sometimes quotients behave well.

\begin{prop} \label{goodquo}   (1)~If $X$ is a locally connected space, all its quotient spaces are  \rr-good. (2)~In particular, if $X$ is a locally connected pseudocompact space then all its continuous images are \rr-good. \end{prop}
\begin{proof} (1)~All the quotients spaces of a locally connected space are locally connected.  (2)~This is by \cite[page~223]{W}.  \end{proof} 

The long line is an example of part (2) of the proposition. 

Note also that every space, \rr-good or not, can be embedded in a direct product of copies of a closed interval, a compact, locally connected (\rr-good) space. 

This section closes with a pair of illustrative examples.  

The space $\Lambda = \beta \R \setminus (\beta \N \setminus \N)$ of \cite[6P]{GJ} is pseudocompact and \rr-good because $\beta \Lambda = \beta \R$  is \rr-good, but it is known not to be locally connected \cite[pp 221,222]{W}. This space will appear again in Example~\ref{Lambda2} and at the end of Section~2.

On the other hand the pseudocompact Tychonoff plank $\T$ is not \rr-good. If it were, $ \beta \T $ would also be \rr-good. However, $\beta \T$ has a clopen subset which is homeomorphic to the one-point compactification of $\N$, a space which is not \rr-good, showing $\beta\T$ is not rr-good by Proposition~\ref{recall}(1).
\\[-.5ex]

\noindent\tbf{2. Product spaces and the \rr-order.}
\setcounter{section}{2}
\setcounter{thm}{0}

In this section the question of \rr-good product spaces will be examined.  It will be easy to see that if a product is \rr-good, so are its factors.  The converse, false in general, will take up much of the section.
\begin{prop}\label{retract}  Suppose $Y$ is a retract of an \rr-good space $X$. Then, $Y$ is \rr-good.  \end{prop}
\begin{proof} Let  $\phi$ and $\psi$ be continuous functions $X\overset{\phi}{\to} Y \overset{\psi}{\to} X$ with $\psi \phi = \mbf{1}_X$ and $f,g\in\cx$. It is easy to see that if $h= f\psi\wrr g\psi$ then $h\phi = f\wrr g$.  \end{proof}

\begin{cor} \label{goodprod}   If $X$ and $Y$ are spaces such that $X\times Y$ is \rr-good, then $X$ and $Y$ are \rr-good. \end{cor}

As already mentioned, the converse is false but there are some cases where there are positive results.
\begin{prop}\label{easyprods}   (1) If $\{X_\alpha\}_{\alpha \in A}$ are locally connected spaces all but finitely many of which are connected then $\prod_{\alpha \in A} X_\alpha$ is \rr-good.  

(2) If $\{X_1, \ldots, X_n\}$ is a finite set of $P$-spaces then $\prod_{i=1}^nX_i$ is \rr-good. \end{prop}

\begin{proof} (1) These products are locally connected and, hence, \rr-good. (2)~A finite product of $P$-space is a $P$-spaces and, hence, \rr-good. \end{proof}

As examples, all euclidean spaces are \rr-good. Other types of \rr-good products will be found at the end of this section.

The following will show that if a space $X$ has enough clopen sets and is \rr-good, then it is basically disconnected. This will play a role later in this section and again in Section~4.

\begin{prop} \label{0-dim}  Let $X$ be a space which has a clopen $\pi$-base. If $X$ is \rr-good then $X$ is basically disconnected.  \end{prop} 

\begin{proof} For $f\in\cx$ it will be shown that $\cl (\coz f)$ is clopen. Since $X$ is \rr-good, $h = \mbf{1} \wrr (\mbf{1} -f)$ exists.  Because $h\rro \mbf{1}$, $h$ is an idempotent and $\coz h=D$ is clopen. Moreover, $h=h^2= h(\mbf{1}-f)$ implies that $hf = \mbf{0}$. When $E\sbq \z(f)$ is clopen, let the idempotent $e$ have cozero set $E$. It follows that $e\rro \mbf{1}, (\mbf{1} -f)$ and, from this, $e\rro h$, giving $E\sbq D$. Hence, $D$ is the unique largest clopen set in $\z(f)$.

If $\cl (\coz f) \ne X\setminus D$ then, because of the clopen $\pi$-base, there would be a non-empty clopen set in $(X\setminus D) \setminus \cl (\coz f)$. This would contradict the fact that $D$ is the maximal clopen set in $\z(f)$. \end{proof} 

A first step in finding examples is to recall two results of Negrepontis.

\begin{prop} \label{neg}  (1)~\cite[Theorem 7.3]{N} For any $P$-space $X$ there exists an extremally disconnected space $Y$ for which $X\times Y$ is not an $F$-space. (2)~\cite[Theorem~6.3]{N} The product of a $P$-space and a compact basically disconnected space is basically disconnected. \end {prop}

\begin{cor} \label{notF}  (1)~If $X$ is a $P$-space and $Y$ is extremally disconnected such that $X\times Y$ is not an $F$-space, then $X\times Y$ is not \rr-good. (2)~If $X$ is a $P$-space and $Y$ is compact and basically disconnected, then $X\times Y$ is \rr-good.   \end{cor}
\begin{proof} (1) If $X\times Y$ were \rr-good, Proposition~\ref{0-dim} would say that it is basically disconnected and, hence, an $F$-space. (2)~Proposition~\ref{neg}~(2) gives the result.  \end{proof}

The case where neither space is a $P$-space can 
also be dealt with as follows.

\begin{thm}\label{prod} Let $X$ and $Y$ be spaces such that each has a clopen $\pi$-base and are not $P$-spaces. The space $X \times Y$ is not \rr-good. \end{thm}

\begin{proof} Every non-empty open set in $X \times Y$ contains a non-empty clopen. If $X \times Y$ were \rr-good it would be basically disconnected by Proposition~\ref{retract}, hence an $F$-space, so one of $X$ and $Y$ would be a $P$-space by \cite[Theorem~p.\,51]{Cu} or by \cite[14Q.1]{GJ} 
\end{proof}

Theorem~\ref{prod} yields families of examples.

\begin{exs} \label{badprods}  If  $X$ and $Y$ are basically disconnected but not 
$P$-spaces, then $X$ and $Y$ are \rr-good but $X\times Y$ is not \rr-good.  As an illustration, $\beta \N \times \beta \N$ is not \rr-good. \end{exs}

Another example of a product of \rr-good spaces that is not \rr-good is found in the next result.  It is of a quite different sort than in Examples~\ref{badprods}, indeed, the factors are connected. The functions needed in the proof are best presented by a description of their graphs.

\begin{thm}\label{hager}   The space $\betatwo$ is not \rr-good.  \end{thm}

\begin{proof} Consider a band of width 2 centred on the diagonal $D =\{(x,x)\mid x\in \R\}$ in $\Rtwo$, bounded by two lines parallel to $D$, $L_1$ above and $L_2$ below. Functions $f,g\in \C(\Rtwo)$ will be defined.

\begin{itemize}
\item[(1)] In the region above and including line $L_1$, $f(x,y) = 3$ and $g(x,y)=2$. 
\item[(2)] In the region below and including $L_2$, $f(x,y)= g(x,y) =0$. 
\item[(3)] Let $L_3$ be the line parallel to $D$, midway between $D$ and $L_1$. On any line $M$ perpendicular to $D$, let $f$ go linearly from 3 to 0 as $(x,y)$ goes from $L_1$ to $L_3$. Similarly, $g$ will go linearly from 2 to 0 on $M$.
\item[(4)] Everywhere below $L_3$ both $f$ and $g$ will be 0 except where indicated below.
\item[(5)] For each $n\in \N$ consider a disk $\Delta_n$ of radius 1/4 around $(n,n)$. The functions $f$ and $g$ will coincide on $\Delta_n$ and their graphs there will be a regular cone of height 1 and centre $(n,n)$. 
\end{itemize}

\vspace{-.7ex}There are several claims to be proved.

\noindent\tbf{Claim 1:} Both $f$ and $g$ extend to $\betatwo$.

As is customary in $\R\times \R$, the first factor is the horizontal axis and the second the vertical one. 

It must be shown that the oscillation condition of \cite[Theorem, page~200]{W} is satisfied so that $f$ and $g$ can be extended to $\beta\R \times \beta\R$. 

It is readily seen that the functions $f$ and $g$ are uniformly continuous because of the repeated patterns along the diagonal.  This means that for every  $\varep >0$ there is $\delta >0$ such that $|f(x,y) -f(u,v)|< \varep$ if $\Vert(x,y), (u,v)\Vert < \delta$. Now fix $(x_0, y_0)$, set $\zeta = (1/\sqrt{2}) \delta$ and consider the vertical lines through $x_0+ m\zeta$ and horizontal lines through $y_0+n\zeta$, with $m,n\in \Z$.  Set $U_1= \bigcup_{m\in \Z}(x_0 +m\zeta, x_0+(m+1)\zeta)$, a union of intervals along the horizontal axis; and $V_1= \bigcup_{n\in \Z}(y_0+ n\zeta, y_0+ (n+1)\zeta)$, a union of intervals along the vertical axis. 

Similarly, construct $U_2$ and $V_2$ using $\{x_0+ (1/2 + m)\zeta\}$ and $\{y_0+ (1/2+ n)\zeta\}$, $m,n \in \Z$.  
Then the open sets $U_1\times V_1$ and $U_2\times V_2$ cover $\R\times \R$ and make a grid satisfying the oscillation conditions. Since this can be done for all $\varep >0$, $f$ and $g$ can be extended to elements of $\C(\beta \R\times \beta \R)$, say $F$ and $G$, respectively.

\noindent\tbf{Claim 2:}  If $H = F\wrr G$ then for $n, m \in \N, n \neq m, H(n,m)=0$. In the case where $ n > m $, this holds because both $f$ and $g$ vanish at $(n,m)$. In the case where $n < m $, this holds because $f$ and $g$ do not agree at $(n,m)$ (Lemma~\ref{fundlem}(2)).

\noindent\tbf{Claim 3:} For $n\in \N$ let $h_n\in \C(\Rtwo)$ be the function whose graph is the cone defined for $(n,n)$.
The function $h_n$ extends to $\betatwo$ because of the 
same grid as used for $f$ and $g$. Let $H_n$ denote its extension to $\betatwo$. Observe that   $H_n \rro F,G$ since $h_n \rro f,g$ on the dense subset $\Rtwo $.

\noindent\tbf{Claim 4:} It follows from Claim~3 that $H(n,n) =1$ for all $n\in \N$ since $H_n(n,n) = 1$ and $H_n\rro H$. This means that $H$, restricted to $\N\times \N$, is the Kronecker delta function, which is shown in \cite[p.\,196]{W} not to extend to $\betaNtwo$. This is a contradiction.
\end{proof}

\begin{cor}\label{higher}  For any $m,n\ge 1$, $\beta (\R^m)\times \beta (\R^n)$ is not \rr-good.  \end{cor}
\begin{proof}  The space $\R$ is a retract of $\R^m$ and, hence, $\beta\R$ is a retract of $\beta(\R^m)$ and, similarly, $\beta \R$ is retract of $\beta (\R^n)$.  From this, $\beta\R \times \beta \R$ is a retract of $\beta(\R^m) \times \beta (\R^n)$. The result follows from Theorem~\ref{hager} and Proposition~\ref{retract}. \end{proof}

It would be interesting to know if $\R\times \beta \R$ is \rr-good or not.  The methods used above do not apply to this space. 

For spaces $X$ and $Y$ it is possible for $\beta (X\times Y)$ to be homeomorphic to $\beta X\times \beta Y$, where the homeomorphism does not fix $X\times Y$ (see \cite[8.18]{W} for such an example). In the case of a homeomorphism, all the spaces $X\times Y$, $\beta (X\times Y)$ and $\beta X\times \beta Y$ are simultaneously \rr-good or none of them is.  The latter possibility is illustrated in the next example.

\begin{ex} \label{Lambda2} There is a connected non-compact \rr-good pseudocompact space whose product with itself is pseudocompact but not \rr-good.  \end{ex}
\begin{proof} The example is the space $\Lambda = \beta \R \setminus (\beta \N\setminus \N)$ mentioned at the end of Section~1. It is \rr-good but $\beta \Lambda \times \beta\Lambda = \beta\R \times \beta \R$ is not \rr-good. On the other hand, $\Lambda\times \Lambda$ is pseudocompact by \cite[Proposition, p.\,203]{W}. Hence, Glicksberg's theorem applies showing that $\beta ( \Lambda \times \Lambda) = \beta \Lambda \times \beta\Lambda$. This is not \rr-good and therefore $\Lambda \times \Lambda $ is not \rr-good either.
\end{proof}

This section ends with some \rr-good products.  Unlike previous examples, the ones to be presented here need not be locally connected or basically disconnected.  A simple lemma, whose proof follows by direct point-wise calculations, will be useful. 

\begin{lem}\label{goodlem}  Let $Y$ be an \rr-good space and $X$ any space. Suppose $f,g \in \C(X\times Y)$. For each $x\in X$, let $f_x, g_x\in \C(Y)$ be given by $f_x(y) = f(x,y)$ and $g_x(y) = g(x,y)$. Set $h_x= f_x\wrr g_x$ and let $h$ be defined by $h(x,y) = h_x(y)$, for all $x,y$. If $h$ is continuous on $X\times Y$ then $h=f\wrr g$.  \end{lem}

Recall that a space $X$ is \emph{weakly Lindel\"of} if for any open cover $\{U_\alpha\}_{\alpha\in A}$, there is a countable subfamily $\{U_{\alpha_n}\}_{n\in \N}$ with $\bigcup_{n\in \N}U_{\alpha_n}$ dense in $X$.  In the following, a result of Comfort, Hindman and Negrepontis (\cite{CHN}) will be crucial.

\begin{thm} \label{LindProp2}   Let $X$ be an arbitrary $P$-space and $Y$ an \rr-good weakly Lindel\"of space. The space $X\times Y$ is \rr-good. \end{thm}

\begin{proof} Fix $f,g\in \C(X\times Y)$. By \cite[Lemma~3.2]{CHN}, each point in $X\times Y$ lies in an open set of the form $U\times V$, $U$ open in $X$, $V$ open in $Y$, such that for $x,x'\in U$ and all $y\in V$, $f_x(y) = f_{x'}(y)$ (the notation is as in the lemma). There are such open sets for $g$ as well and, by taking intersections, it may be assumed that these open sets work for both functions. Moreover, since $X$ is a $P$-space, it may also be assumed that $U$ is clopen. An open set $U\times V$ where $U$ is clopen and the \cite{CHN} properties hold for both $f$ and $g$ will be here called a \emph{tile}. 

Fix $p\in X$. For each $y\in Y$ there is a tile $U\times V$ with $(p,y)\in U\times V$.  Hence, the set of tiles $\{U_\beta \times V_\beta\}_{\beta \in B}$, with $p\in U_\beta$, is such that $\bigcup_{\beta \in B} V_\beta = Y$. By the weakly Lindel\"of property, there is a countable subset $\{V_{\beta_n}\}_{n\in \N}$ whose union, $V_p$, is dense in $Y$. Put $A_p=\bigcap_{n\in \N}U_{\beta_n}$. Since $X$ is a $P$-space, $A_p$ is clopen. For all $x\in A_p$, $f_x= f_p$ are equal on the dense open set $V_p$. Hence, $f_x= f_p$ on $Y$. Similarly $g_x= g_p$ on $Y$.  Put $h_p = f_p\wrr g_p$ and notice that $h_p= f_x\wrr g_x$, 
for all $x\in A_p$.

Now consider $p, p'\in X$. If $A_p\cap A_{p'}\ne \vide$, then $h_p= h_{p'}$. Indeed, for $x\in A_p\cap A_{p'}$, $f_x=f_p = f_{p'}$ and $g_x= g_p = g_{p'}$.  Define $h(x,y) = h_p(y)$ whenever $x\in A_p$. This is well-defined and continuous on all the elements of the open cover of $\{A_p\times Y\}_{p\in X}$. Hence, $h = f\wrr g$ by Lemma~\ref{goodlem}.
  \end{proof}

There are many examples of \rr-good Lindel\"of spaces $Y$ which can be used in Proposition~\ref{LindProp2}, for example $\R$. For any non-discrete $P$-space $X$,  $X\times \R$ is \rr-good but neither locally connected nor basically disconnected. An example where the \rr-good space $Y$ is weakly Lindel\"of but not Lindel\"of is $\Lambda$, described at the end of Section~1. Another is found in  \cite[Example~2, p.\ 237]{LR}.  \\[-.5ex]

\setcounter{section}{3}  \setcounter{thm}{0}  
\noindent\tbf{3. A partition of $\R^2$ into two dense subspaces, one \rr-good: this is impossible in $\R$.}

The first thing to note is that if the real line $\R= A\cup B$, with $A\cap B=\vide$ and $A$ and $B$ both dense, then neither $A$ nor $B$ is \rr-good. This is a consequence of the following.
\begin{lem} \label{nogood} \cite[Proposition~3.6]{BR1} Suppose, in a space $X$, there is a sequence $\{D_n\}_{n\in \N}$ of pairwise disjoint clopen sets such that $U=\bigcup_{n\in \N}D_n$ is not closed and there is $x\in \bd U$ (the boundary or frontier) such that every neighbourhood of $x$ meets all but finitely many of the $D_n$.  Then, $X$ is not \rr-good.  \end{lem}

To use the lemma in the case of $A$ and $B$ in $\R$, it suffices to take a convergent increasing sequence $\{a_n\}_{n\in\N}$ in, say, $A$ and intersperse it with a sequence from $B$. 
 
 Subsets of $\R^2$ will now be constructed to show a quite different situation in the plane. 
 
\begin{defn} \label{ml}  A line $y = mx+b$ in $\R^2$ is 
called \emph{matched} if $m, b\in \Q$ and $m\ne 0$. The graph of such a 
line is denoted $L_{m,b}$.\end{defn}

\begin{lem} \label{lines}  Consider a matched line 
$L_{m,b}$ in $\R^2$ given by $y = mx+b$, where $m\ne 0$ and $m,b \in 
\Q$.  Then if $(p,q)\in L_{m,b}$, both $p,q \in \Q$ or both are 
irrational. \end{lem} \begin{proof} If $x\in \Q$ then $y =mx+b\in \Q$. 
If $x\notin \Q$ then $y =mx+b \in \Q$ would imply $mx\in \Q$, but $m\in 
\Q$ and $x\notin \Q$, which is impossible. \end{proof}

 It is also useful to note that if $(a,b)$ and $(c,d)$ are such that $a,b,c,d\in \Q$, $a\ne c$, then the line joining these points is a matched line.

\begin{thm} \label{pwc}  Consider the following two subsets of $\R^2$: $$B=\bigcup_{m,b\in \Q, m\ne 0}L_{m,b}\;\text{and}\;A= \R^2\setminus B\;.$$ Then,

(1) $B$ is dense in $\R^2$,  locally connected and, hence, \rr-good.

(2) $A$ is dense in $\R^2$ and has a basis of clopen sets.  It is not \rr-good. \end{thm}

\begin{proof} (1) Since any open set in $\R^2$ contains points where both coordinates are rational, $B$ is dense in $\R^2$. Notice that $B$ also contains points $(a,b)$ where both $a$ and $b$ are irrational, but not all such points. 

Consider a point $(a,b)$ in $B$ and an open disk $C$ with centre $(a,b)$. Suppose that $U$ and $V$ are open sets of $\R^2$ such that $U\cup V\spq C\cap B$, $U\cap V \cap C \cap B= \vide$, $U\cap C\cap B\ne \vide$ and $V\cap C\cap B\ne \vide$. In other words assume that there is 
a partition of $C\cap B$.  Choose points $(p,q)\in U\cap C\cap B$ and $(u,v)\in V\cap C\cap B$, $p,q,u,v\in \Q$, $u\ne p$. The line segment joining these two points will lie in $C\cap B$ but this line segment is connected in $\R^2$, which is impossible. Hence, $B$ is locally connected and, hence, \rr-good by \cite[Theorem~3.5(1)]{BR1}. (It can be seen that $B$ is even arcwise connected.) 

(2) The set $A$ contains all points $(a,b)$ where one coordinate is rational and the other irrational, as well as some points where both coordinates are irrational. This shows that $A$ is dense in $\R^2$.  Moreover, for any $(a,b)\in A$ and any open disk $C$ with centre $(a,b)$ 
there is a quadrilateral inside $C$ containing $(a,b)$ bounded by matched lines. The interior of such a quadrilateral, intersected with $A$, is a clopen set in $A$.

Since $A$ has a basis of clopen sets and has convergent sequences, it is not \rr-good by Lemma~\ref{nogood}. \end{proof}

There are similar constructions in $\R^n$, $n>2$. \\[-.5ex]

\setcounter{section}{4}\setcounter{thm}{0}
\noindent\tbf{4. Some separation properties and $\cx$.}

Two sorts of reduced rings will make an appearance in this section. The definitions are recalled here and, in the case of $\cx$, the corresponding topological notions will follow. 

\begin{defn} \label{wBawB} (1)~A ring $R$ is called \emph{weakly Baer} or \emph{wB} if, for each $r\in R$, $\ann r$ is generated by an idempotent $e=e^2$.  (2)~A ring $R$ is called \emph{almost weakly Baer} or \emph{awB} if, for each $r\in R$, $\ann r$ is generated by a set of idempotents. \end{defn} 

In the literature the names ``pp-ring'' and ``almost pp-ring'' are also used for wB and awB rings, respectively. 

The first thing to note is the following. 

\begin{lem} \label{awBvswB}  \cite[Theorem~2.6]{BR1} An awB ring is \rr-good if and only if it is wB. \end{lem}

Not all awB rings are wB. 
\begin{ex} \label{NRex} \cite[Example~3.2]{NR} The ring $\C(\beta\N\setminus \N)$ is awB but not wB. \end{ex}
 
Even though awB rings need not be \rr-good, a topological description of them nicely parallels that for wB rings of the form $\cx$, and is given here. 
 
The equivalence of the first two statements in the following is mentioned in \cite{NR} but is also proved here.

\begin{prop} \label{wBchar}  The following three statements about a space $X$ are equivalent. (1)~$X$ is basically disconnected; (2)~$\cx$ is a wB ring; and (3)~if $U$ is a cozero set and $V$ an open set with $U\cap V=\vide$ then $U$ and $V$ can be separated by a clopen set.  \end{prop}

\begin{proof} (1) $\Rightarrow$ (2): Consider $f\in \cx$ and let $D= X\setminus (\cl(\coz f))$, a clopen set, and $e=e^2\in\cx$ such that $\coz e=D$. For any $g\in \ann f$, $ge =g$ and $fe=0$. Hence, $\ann f= e\cx$. (2) $\Rightarrow $ (3): Let $U=\coz f$ and $V$ be open with $U\cap V=\vide$. Since $\ann f= e\cx$ for some $e=e^2$, the clopen set $D = \coz e$ is such that $\coz f = U\sbq X\setminus D$. For every $g\in \cx$ with $\coz g\sbq V$, $fg=\mbf{0}$ implying that $\coz g\sbq D$. Thus, the clopen set $D$ separates $U$ and $V$.
(3) $\Rightarrow$ (1): If $U =
\coz f$, put $V= \inte (X\setminus U)$. There is a clopen set $D$ with $\coz f\sbq D$ and $V\sbq X\setminus D$. It follows that $\cl \,U= D$.  \end{proof}

The equivalence of (1) and (2) in the next result was obtained in \cite[Theorem~2.4]{AE}, but the proof here is more direct. $U$-spaces were introduced in \cite{GH}; they are spaces $X$ such that, for each $f\in \cx$, there is a unit $u\in \cx$ with $f= |f|u$. 

\begin{prop} \label{awBchar}  The following statements for a space $X$ are equivalent. (1)~$X$ is a $U$-space; (2)~$\cx$ is an awB ring; and (3)~if $U$ and $V$ are cozero sets with $U\cap V=\vide$ then $U$ and $V$ can be separated by a clopen set.
\end{prop}
\begin{proof} (1) $\Rightarrow$ (2): Let $\mbf{0}\ne f,g\in \cx$ with $fg = \mbf{0}$.  Replace $f$ by $k = -|f|$ and $g$ by $l=|g|$; the cozero sets do not change.  There is a unit $u$ such that $k+ l = |k+l|u= (-k+l)u$. Hence, for $x\in \coz k$, $u(x) = -1$ and for $x\in \coz l$, $u(x) =1$. Since $u$ is a unit, there is a clopen set $D$ such that for $x\in D$, $u(x)>0$ and for $x\notin D$, $u(x)<0$. From this, $\coz k = \coz f\sbq X\setminus D$ and $\coz l= \coz g\sbq D$. Put $e=e^2$ with $\coz e = D$. Then, $fe =\mbf{0}$ and $g=eg$, showing that $\ann f$ is generated by idempotents.

 (2) $\Rightarrow$ (3): Let $U = \coz f$ and $V= \coz g$ be such that $U\cap V= \vide$. The product $fg=\mbf{0}$. Since $\cx$ is awB there are $e_i = e_i^2$ and $l_i\in \cx$, $i = 1,\ldots, k$, with each $e_i$ such that $fe_i=\mbf{0}$ and $g= \sum_{i=1}^k e_il_i$. Since $D=\bigcup_{i=1}^k \coz e_i$ is clopen, there is $e=e^2$ with $\coz e=D$. From this, $fe=\mbf{0}$ and $g=ge$.   The clopen set $\coz e$ separates $U$ and $V$. 

(3) $\Rightarrow$ (1): It must be shown that for any $f\in \cx$ there is a unit $u$ with $f= |f|u$. If $f$ does not change sign in $\coz f$, the unit can be $\pm\mbf{1}$. Otherwise,  let $f^+$ be defined by $f^+(x) = f(x)$ if $x\in \coz f$ and $f(x)>0$, $f^+(x) = 0$ for other $x$. Similarly, $f^-$ is defined. Since $\coz f^+ \cap \coz f^- = \vide$, there is a clopen set $D$ with $\coz f^+\sbq D$ and $\coz f^-\sbq X\setminus D$. Let $e=e^2$ be such that $\coz e=D$ and $u= e -(\mbf{1}-e)$, a unit. From this, $f = f^+ + f^- = f^+u - f^-u = |f|u$. 
\end{proof}

\setcounter{section}{5}  \setcounter{thm}{0}  

\noindent\tbf{5. A sufficient but not a necessary condition for \rr-good.}

We begin by recalling the definition of the B-property from 
\cite[Definition 3.3]{BR1}. It is a sufficient condition for a space to be \rr-good (\cite[Corollary~3.4]{BR1}).  It is implied by local connectedness.  However, it is known not to be a necessary condition; a topic expanded upon here.  

\begin{defn}\label{b-prop}  In a space $X$ let $\{U_\alpha\}_{\alpha \in A}$ be any family of non-empty cozero sets in $X$ with the following property: for $\alpha \ne \beta$ in $A$, $(\bd U_\alpha) \cap U_\beta = \vide$. The space $X$ is said to satisfy the 
\emph{B-property} (for \emph{boundary property}) if the following holds for each such family of cozero sets. Let $z\in \bd(\bigcup_{\alpha\in A}U_\alpha)$. For every neighbourhood $N$ of $z$ there is $\beta \in A$ such that $N\cap \bd U_\beta \ne \vide$. \end{defn}

The motivation for this definition is as follows: Suppose in $\cx$ that, for $f,g\in \cx$, there are non-zero \rr-lower bounds $\{h_\alpha\}_{\alpha \in A}$ for $f$ and $g$.  Then, by Lemma~\ref{fundlem}, the set $\{\coz h_\alpha\}_{\alpha \in A}$ satisfies the demands of Definition~\ref{b-prop}.

The purpose here is to find \emph{connected} \rr-good spaces without the B-property.  Before doing that, the next proposition shows that, at the other extreme, it is easy to find basically disconnected spaces without the B-property.  

\begin{prop} \label{pibase}  If $X$ is a space that has the B-property then each union of clopen sets is clopen. If, in addition, $X$ has a clopen $\pi$-base, it is discrete.   \end{prop}

\begin{proof} Any set $\{U_\alpha\}_{\alpha \in A}$ of clopen sets satisfies the conditions of Definition~\ref{b-prop}.  Set $U=\bigcup_{\alpha \in A}U_\alpha$. If $x\in \bd U$, any neighbourhood of $x$ would meet $\bd U_\alpha$, for some $\alpha$.  However, $\bd U_\alpha = \vide$ and, hence, $\bd U=\vide$. For the second part, any open $V$ in $X$ has a union of clopen sets dense in it.  From the first part, $V$ is clopen.   \end{proof}

Any basically disconnected space $X$ which is not discrete is \rr-good and does not have the B-property. 

The next proposition is the key tool for the construction of connected examples. 

\begin{prop} \label{beta}   Let $X$ be a space which is not compact. Let $\{U_\alpha\}_{\alpha\in A}$ be an infinite family of pairwise disjoint non-empty cozero sets of $X$.  For $\alpha \in A$, let $f_\alpha$ be such that $f_\alpha \in C(X)$ such that $\coz f_\alpha = U_\alpha$, for all $x\in X$, $0\le f_\alpha(x) \le 1$ and for some $k_\alpha \in U_\alpha$, $f_\alpha(k_\alpha) = 1$. Assume that these data also satisfy the following properties:

(i) for $\alpha \ne \beta$, $(\bd_XU_\beta) \cap U_\alpha =\vide$,

(ii) $K_\alpha =\cl_XU_\alpha$ is compact for all $\alpha \in A$,

(iii) the function $f$ defined by $f(x) = f_\alpha(x)$ for $x\in U_\alpha$ and $f(x) =0$ if $x\notin U=\bigcup_AU_\alpha$ is continuous on $X$,

(iv) $K=f\inv (\{1\})$ is not compact in $X$.\\ 
Then, $\beta X$ does not have the B-property.  
\end{prop}

\begin{proof} The condition (ii) says that $K_\alpha$ is compact and so it and $X^*= \beta X \setminus X$ are completely separated. There is $u_\alpha\in \C(\beta X)$ such that $u_\alpha \res_{K_\alpha}$ is constantly 1 and $u_\alpha\res_{X^*}$ is constantly 0.  Now $f_\alpha$ can be extended to $\beta f_\alpha\in \C(\beta X)$.  The product $u_\alpha \cdot \beta f_\alpha$ coincides with $f_\alpha$ on $U_\alpha$ and is 0 elsewhere. This shows that $U_\alpha$ is a cozero set in $\beta X$. Moreover, $\bd_XU_\alpha = \bd_{\beta X}U_\alpha$ because $K_\alpha$ is compact and, hence, also closed in $\beta X$. 

It follows that $\{U_\alpha\}_A$ is a family of cozero sets in $\beta X$ which satisfies the condition to test for the B-property. 

Since $K$ is closed and not compact in $X$, $P= (\cl_{\beta X} K)\cap X^*\ne \vide$.  Now, extend $f$ to $\beta f$. 

It follows that $(\coz \beta f)\cap X = U$ but also $P\sbq \coz \beta f$, since for $p\in P$, $\beta f(p) =1$.

Notice that any $p\in P$ is in $\bd_{\beta X}U$ since any neighbourhood $N$ of $p$ with $N\sbq \coz \beta f$ will meet $X$ and thus $ N\cap X\sbq U$. However, any such $N$ will not meet any $\bd_X U_\alpha = \bd_{\beta X}U_\alpha$, contradicting the B-property. 
\end{proof}

\begin{cor} \label{beta-rr}  Suppose that $X$ satisfies the conditions of Proposition~\ref{beta} and that $X$ is \rr-good.  Then, $\beta X$ is \rr-good and does not have the B-property. Moreover, $\beta X$ is not locally connected.  \end{cor} 

\begin{proof} Since $X$ is \rr-good, so is $\beta X$.  Then Proposition~\ref{beta} says that $\beta X$ does not have the B-property. If $\beta X$ were locally connected, it would have the B-property.  \end{proof}

\begin{ex}\label{b-reals}  The connected space $\beta\R$ is \rr-good and does not have the B-property.  \end{ex}

\begin{proof} The space $\R$ is \rr-good.  The cozero sets needed in the proposition can be taken to be the intervals $\{(n,n+1)\}_{n\in \Z}$ and, hence, Corollary~\ref{beta-rr} applies.  \end{proof}

Similarly, for any euclidean space $\R^n$, $\beta\R^n$ is a connected \rr-good space which does not have the B-property.  \\[-.5ex]

\end{document}